\documentclass[reqno,centertags,11pt]{amsart}
\usepackage{verbatim}
\usepackage{tikz, enumerate, array, amssymb, amsmath}
\usepackage{graphicx, epsfig }
\usepackage{amsthm}
\usepackage{amsbsy}
\usepackage{amsfonts}
\usepackage[colorlinks]{hyperref}

\usepackage{mathtools}

\newcommand{\Hmm}[1]{\leavevmode{\marginpar{\tiny%
$\hbox to 0mm{\hspace*{-0.5mm}$\leftarrow$\hss}%
\vcenter{\vrule depth 0.1mm height 0.1mm width \the\marginparwidth}%
\hbox to 0mm{\hss$\rightarrow$\hspace*{-0.5mm}}$\\\relax\raggedright #1}}}

\newcommand{\R}{{\mathbb{R}}}

\newcommand{\C}{{\mathbb{C}}}

\newcommand{\Z}{{\mathbb{Z}}}

\newcommand{\f}{\frac}

\newcommand{\beq}{\begin{equation}}
\newcommand{\eeq}{\end{equation}}
\newcommand{\bdm}{\begin{displaymath}}
\newcommand{\edm}{\end{displaymath}}
\newcommand{\ba}{\begin{align}}
\newcommand{\ea}{\end{align}}
\newcommand{\bpf}{\begin{proof}}
\newcommand{\epf}{\end{proof}}

\newcommand{\la}{\langle}
\newcommand{\ra}{\rangle}

\newcommand{\vphi}{\varphi}

\newcommand{\dav}{{d_{\mathrm{av}}}}

\allowdisplaybreaks

\newcommand{\calC}{\mathcal{C}}

\newcommand{\calF}{\mathcal{F}}

\newtheorem{theorem}{Theorem}
\newtheorem{proposition}[theorem]{Proposition}
\newtheorem{lemma}[theorem]{Lemma}

\theoremstyle{definition}

\newtheorem{remark}[theorem]{Remark}
\newtheorem{remarks}[theorem]{Remarks}

\newcounter{theoremi}[theorem]

\numberwithin{theorem}{section}
\numberwithin{equation}{section}

\newcounter{assumptions}

\newcounter{smalllist}

\newcounter{listi}
\newenvironment{theoremlist}{\begin{list}{{\rm(\roman{listi})}}{%
\setlength{\topsep}{0mm}\setlength{\parsep}{0mm}\setlength{\itemsep}{0mm}%
\setlength{\labelwidth}{1.5em}\setlength{\leftmargin}{1.7em}\usecounter{listi}%
}}{\end{list}}

\newcounter{smallenum}

\begin{document}

\title[Continuum Limit]{Continuum limit related to dispersion managed nonlinear Schr\"{o}dinger equations}
\author[M.--R. Choi, Y.--R. Lee]{Mi--Ran Choi$^\dag$ and Young--Ran Lee$^\ddag$}

\address{$^\dag$ Research Institute for Basic Science, Sogang University, 35 Baekbeom--ro (Sinsu--dong),
    Mapo-gu, Seoul 04107, South Korea.}%
\email{rani9030@sogang.ac.kr}

\address{$^\ddag$ Department of Mathematics, Sogang University, 35 Baekbeom--ro (Sinsu--dong),
    Mapo--gu, Seoul 04107, South Korea.}%
\email{younglee@sogang.ac.kr}


\date{\today}

\begin{abstract}
 We consider the dispersion managed nonlinear Schr\"odinger equation with power-law nonlinearity and its discrete version of equations with step size $h\in(0,1]$. We prove that the solutions of the discrete equations strongly converge in $L^2(\R)$ to the solution of the dispersion managed NLS as $h\to 0$ after showing the global well-posedness of the discrete equations.
\end{abstract}

\maketitle

\section{Introduction }\label{introduction}

In this paper, we consider the dispersion managed nonlinear Schr\"odinger equation (NLS) with power-law nonlinearity
\beq \label{eq:NLS}
i\partial_t u + \dav \partial_x^2 u +\int_0^1T_r^{-1} (|T_r u|^{p-1} T_r u) dr=0\,,
\eeq
where $u=u(x,t):\R \times \R \to \C$, $\dav\in \R$, and $p>1$.
Here, $T_r=e^{ir\partial_x^2}$ is the solution operator of the Schr\"odinger equation, that is, $w(\cdot, r)=T_r\vphi$ is the solution of
 \[
 i\partial_r w + \partial_x^2 w =0,\quad w(\cdot,0)=\vphi.
 \]
This equation, the so-called Gabitov-Turitsyn equation, first appears in \cite{Gabitov:96, Gabitov:96b} as an averaged equation for NLS with a varying coefficient of $t$, see, e.g., \cite{CL22, ZGJT01} for verification of the averaging process. Such NLS models propagation of pulses in fiber-optics communication systems where the local dispersion varies periodically with alternating signs along the cable, see, e.g., \cite{SS} . Here, $x$  corresponds to the (retarded) time, $t$ the distance along the cable, and $\dav$ the average of the dispersion along the cable.
The technique of strong dispersion management, via rapidly and strongly varying dispersion, was invented to balance the effects of dispersion and nonlinearity. This technique generates stable soliton-like pulses (breather type solutions) and was successful in transferring data at ultra-high speeds over intercontinental distances, see, e.g., \cite{Ablowitz:98,  Gabitov:96, Gabitov:96b}.
The dispersion managed NLS is relatively well-understood in various contexts, for example, the existence of ground states in \cite{CHLT, HL,  MR2033144, ZGJT01}, some properties of ground states in \cite{MR2770579, MR2472020, Stanislavova}, and orbital stability of the set of ground states as well as global well-posedness in \cite{CHLW,HKS2015}.

\bigskip

As a discretization of equation \eqref{eq:NLS}, for each $h>0$, we consider 
\beq \label{eq:dNLS}
i\partial_t u_h + \dav \Delta_hu_h +\int_0^1T_{h,r}^{-1} (|T_{h,r} u_h|^{p-1} T_{h,r} u_h) dr=0,
\eeq
where $u_h=u_h(x,t):h\Z\times \R \to \C$, $h\Z=\{x=hm: m\in \Z\}$, and the discrete Laplacian defined by
\[
(\Delta_h f)(x) =\f{f(x+h)+ f(x-h)-2f(x)}{h}
\]
for all $x\in h\Z$. Here, $T_{h,r}= e^{ir\Delta_h}$ is the solution operator for the discrete Schr\"odinger equation, that is, $w(\cdot, r)=T_{h,r}\vphi$ is the solution of
 \begin{equation}\label{eq:discrete_Schrodinger}
 i\partial_r w + \Delta_h w =0,\quad w(\cdot,0)=\vphi.
 \end{equation}
In nonlinear optics, equation \eqref{eq:dNLS} is related to the diffraction managed discrete NLS that is a model for an array of coupled nonlinear waveguides, where the waveguides' diffraction is periodically and strongly varied, see, e.g., \cite{PhysRevLett.87.254102,PhysRevE.65.056618,ABLOWITZ2003276}. Here, $x$ corresponds to the location of the waveguides, $t$ the distance along the waveguides, and $\dav$ the average of the diffraction along the waveguides.
This equation with $h=1$ had first been rigorously studied in \cite{Moeser05, Panayotaros}. The existence of ground states and orbital stability of the set of ground states were proved in \cite{Moeser05, Panayotaros} for positive average diffraction and \cite{Stanislavova2} for zero average diffraction.
See \cite{CHL2017, HL, Stanislavova2} for the decay and smoothness of ground states.
In another view point, the discrete NLS \eqref{eq:dNLS} can naturally be considered as a numerical approximation of the dispersion managed NLS \eqref{eq:NLS}.

Our main interest is how to connect the solutions of \eqref{eq:NLS} and \eqref{eq:dNLS}. We first show the well-posedness for the Cauchy problem of \eqref{eq:dNLS} which is proved  in \cite{Moeser05} only when $p=3$.
\begin{theorem}[Global Well-Posedness] \label{thm:globalwellposedness}
Let $h\in (0, 1]$ and $p>1$. For the initial datum $\vphi_h\in L^2_h(h\Z)$, there exists a unique solution $u_h\in \calC(\R,L^2_h(h\Z) )$ of \eqref{eq:dNLS}.
Moreover, $u_h$ continuously depends on the initial data and it conserves the mass and the energy, that is,
\beq
\|u_h(t)\|^2_{L^2_h}=\|\vphi_h\|^2_{L^2_h} \notag \quad\mbox{and}\quad
 E(u_h(t))= E(\vphi_h)
 \eeq
 for all $t\in\R$, where the energy is given by
 \beq \label{energy}
 E(f):=\f \dav 2 \| D_h^+ f\|_{L^2_h}^2 - \f{1}{p+1}\int_0 ^1 \|T_{h,r} f\|_{L^{p+1}_h}^{p+1} dr\notag
 \eeq
 for $f\in L^2_h(h\Z)$.
\end{theorem}
Here, the norm $\| \cdot \|_{L_h ^p}$, $1\leq p < \infty$, is given by
\[
\|f\|_{L_h^p} :=h ^{\f{1}{p}}\|f\|_{l^p(h\Z)}=
\left\{h \sum_{x\in h\Z} |f(x)|^p \right\}^{1/p}
\]
and
$D_h^+$ denotes the forward difference operator
\beq\label{diff}
(D_h^+f)(x):= \f{f(x+h)-f(x)}{h}
\eeq
for $x \in h\Z$.

\bigskip

Now we consider the continuum limit of discrete version \eqref{eq:dNLS} of dispersion managed NLS \eqref{eq:NLS} as $h\to 0$. It is naturally expected that the solutions of  \eqref{eq:dNLS} converge to the solution of \eqref{eq:NLS} in some sense.
For a more precise statement, we give some notations.
Given a complex-valued function $f$ in $L^1_{loc}(\R)$, define its discretization $f_h:h\Z \to \C$ by
\beq
f_h(x):=\f{1}{h} \int_{x}^{x+h} f(x')dx'\notag
\eeq
for all $ x\in h\Z$.
Conversely, we define the linear interpolation operator $p_h$ mapping a function $f:h\Z \to\C$ to a function on $\R$ by
\beq
(p_h f)(x):=f(x_m) + \f{f(x_m + h)-f(x_m)}{h}(x-x_m)\notag
\eeq
for all $x\in [x_m,  x_m+ h)$, $x_m=hm\in h\Z$.

The continuum limit for discrete NLS was first studied in \cite{KLS2013}. They proved that solutions of one-dimensional cubic discrete NLS with long-range lattice interactions weakly converge to the solution of the corresponding (fractional) NLS as $h\to0$. In \cite{HY2019}, the authors improved the result in \cite{KLS2013} showing that solutions of discrete fractional NLS with power-law nonlinearities  strongly converge in $L^2(\R^d), d=1,2,3,$ to the solution of the corresponding continuum equation. Moreover, they gave a precise rate of strong convergence. Recently, they managed to extend the result to the discrete NLS on a periodic cubic lattice, see \cite{HKNY}. The continuum limit for discrete NLS with memory effect was shown in \cite{Grande} adapting the method in \cite{HY2019}.

We prove that solutions of discrete version \eqref{eq:dNLS} of dispersion managed NLS \eqref{eq:NLS} strongly converge in $L^2(\R)$ to the solution of \eqref{eq:NLS} as $h\to0$.
\begin{theorem}[Continuum Limit] \label{thm:continuum limit}
Let $h\in (0,1]$ and $p$ satisfy
\beq\label{ass:p}
\begin{aligned}
 \left\{
                                 \begin{array}{ll}
                                   \displaystyle 1<p<9& \hbox{if} \;\;\dav> 0
                                   \\[0.5ex]
                                  \displaystyle  p>1 &  \hbox{if} \;\; \dav<0
                                  \\[0.5ex]
                                  \displaystyle  1<p<5 &  \hbox{if} \;\; \dav=0.
                                 \end{array}
                               \right.
\end{aligned}
\eeq
Given initial datum $\vphi\in H^1(\R)$, let $u\in \calC(\R, H^1(\R))$ be the global solution to dispersion managed NLS \eqref{eq:NLS} and $u_h$ the global solution to discrete equation \eqref{eq:dNLS} whose initial datum $\vphi_h$ is the discretization of $\vphi$, for each $h\in (0.1]$. Then
  there exist positive constants $A$ and $B$, independent of $h$, such that for all $0<T<T^*$,
\[
\sup_{t\in [-T, T]}\|p_h u_h(t)  -u(t)\| _{L^2} \leq A h^{\f  1 2} e^{BT}\,,
\]
where $T^*=\infty$ if $\dav\neq0$ and  $T^*=\f{2}{p-1}(\|\vphi\|_{L^2} \|\vphi'\|_{L^2})^{-\f{p-1}{2}}$ if $\dav=0$.
\end{theorem}

\begin{remarks}
\begin{theoremlist}
 \item It is remarkable that Theorem \ref{thm:continuum limit} holds even for $5\leq p <9$ when $\dav>0$ in contrast to the classical focusing NLS, see \cite{HY2019}.
 \item 
The global well-posedness  in $H^1(\R)$ of \eqref{eq:NLS} for the Kerr nonlinearity, i.e., $p=3$, was first proven in \cite{ZGJT01}, see  \cite{AK} for the $H^s$ theory, $s \ge 0$. More general nonlinearities including even saturated nonlinearities were considered in \cite{CHLW}. As a special case of \cite{CHLW}, in the case of power-law nonlinearities, \eqref{eq:NLS} is globally well-posed  in $H^1(\R)$ under the condition \eqref{ass:p}. Indeed, in \cite{CHLW}, the global well-posedness in $L^2(\R)$ of \eqref{eq:NLS} when $\dav=0$ was treated, however, one can easily show the global existence of the solution in $H^1(\R)$ by the blow-up alternative.
 \item
In the case of $\dav=0$, the above result holds locally, depending on $p$ and the initial datum, in contrast to the case of $\dav\neq 0$ where the result holds for all finite $T>0$.
Such a difference is caused by Propositions \ref{prop:uniform} and \ref{prop:uniform2} which show the uniform $H^1_h$-bounds for solutions $u_h$ for $\dav\neq 0$ and $\dav=0$, respectively.
\end{theoremlist}

\end{remarks}

%

\bigskip

The paper is organized as follows:
  In Section \ref{sec:preliminary}, we introduce some notations and establish some useful estimates.
  Section \ref{sec:global well-posedness} is devoted to the global well-posedness, Theorem \ref{thm:globalwellposedness}.
  In Section \ref{sec:continuum limit}, we first prove the uniform $H_h^1$-bound for the solutions which is the key ingredient of Theorem \ref{thm:continuum limit} and finally give the proof of Theorem \ref{thm:continuum limit}.

\medskip

\section{Preliminary Estimates}\label{sec:preliminary}

We start by introducing some notations.
The Hilbert space $L^2_h(h\Z)$ is equipped with the inner product
$$
\la f, g \ra_{L^2_h}:= h\sum_{x\in h\Z} f(x) \overline{g(x)}\, .
$$
For $1\leq p \le \infty$, the space $L_h ^p(h\Z)$ is the Banach space with norm
\[
\begin{aligned}
\|f\|_{L_h^p} :=h ^{\f{1}{p}}\|f\|_{l^p(h\Z)}=  \left\{
                                 \begin{array}{ll}
                                   \displaystyle \left\{h \sum_{x\in h\Z} |f(x)|^p \right\}^{1/p} & \hbox{if} \;\; 1\leq p< \infty
                                   \\[0.5ex]
                                  \displaystyle \sup_{x\in h\Z}|f(x)| & \hbox{if} \;\; p=\infty.
                                 \end{array}
                               \right.
\end{aligned}
\]
The Fourier transform of a function $f\in L^2_h(h\Z)$ is defined by
\[
\hat{f} (\xi)= (\calF_hf)(\xi):=\f{h}{\sqrt{2\pi}} \sum_{x\in h\Z}f(x)e^{-ix\xi}
\]
for $\xi \in \f{1}{h}[-\pi,\pi]=[-\f{\pi}{h}, \f{\pi}{h}]$
and its inversion formula is given by
\[
\check{f}(x)=(\calF_h^{-1}f) (x):= \f{1}{\sqrt{2\pi}} \int_{-\f{\pi}{h}}^{\f{\pi}{h}} f(\xi) e^{ix\xi} d\xi
\]
for $x\in h\Z$.
For $f, g \in L^2_h(h\Z)$, the Parseval identity yields
\[
 \int_{-\f{\pi}{h}}^{\f{\pi}{h}} \hat{f}(\xi)\overline{\hat{g}(
\xi)} d\xi =h \sum_{x\in h\Z} f(x)\overline{g(x)}.
\]
Fix $h>0$. For any function $f\in L^2_h(h\Z)$, we define the $H^s_h$ and $\dot{H}^s_h$ norms of $f$, $s\in\R$, by
\[
\|f\|_{H^s_h}=  \left(\int_{-\f{\pi}{h}}^{\f{\pi}{h}}(1+|\xi|^{2})^s |\hat{f}(\xi)|^2 d\xi \right)^{1/2}
\]
and
\beq\label{def:H}
\|f\|_{ \dot{H}^s_h}= \left( \int_{-\f{\pi}{h}}^{\f{\pi}{h}}|\xi|^{2s} |\hat{f}(\xi)|^2 d\xi\right)^{1/2}\notag
\eeq
which are finite, see \cite[Proposition 1.2]{HY2018} for the proof. We will use equivalence among the following norms when $s=1$. A simple calculation gives that $\|\cdot\|_{\dot{H}^1_h}$ and $\|D_h^+ \cdot\|_{L^2_h}$ are equivalent, more precisely,
\beq \label{eq:eq}
\f{2}{\pi}\|f\|_{\dot{H}^1_h} \leq \|D^+_h f\|_{L^2_h}\leq \|f\|_{\dot{H}^1_h},
\eeq
where the forward difference operator $D^+_h$ is given in \eqref{diff}.
Therefore, we can use the norm
\beq\label{eq:equivalent norm}
\left(\|f\|^2_{L^2_h}+\|D^+_h f\|^2_{L^2_h}\right)^{1/2}
\eeq
of $f \in L^2_h(h\Z)$ instead of $\|f\|_{H^1_h}$. 
Noting
\begin{equation}\label{eq:norm-difference}
\|D^+_hf\|_{L^2_h}\leq 2 h^{-1} \|f\|_{L^2_h}\notag
\end{equation}
for any $f\in L^2_h(h\Z)$ deduces
\begin{equation}\label{eq:relation L^2 H^1}
\|f\|_{L^2_h} \le \|f\|_{H^1_h}\leq C_h\|f\|_{L^2_h}\, ,
\end{equation}
where $C_h$ is dependent on $h$. 


We gather some elementary inequalities following from \cite{HY2019, HY2018, KLS2013} and \eqref{eq:eq}. We use the notation $f\lesssim g$ when there exists a positive constant $C$, independent of $h$, such that $f \leq Cg$.
\begin{lemma}\label{lem:G-N} 
 Let $h\in (0,1]$.
 \begin{theoremlist}
 \item
 (Gagliardo-Nirenberg Inequality)
If $2<q \le \infty$ and $\theta= \f 1 2 - \f 1 q $, then
\beq \label{ineq:G-N}
\|f\|_{L^q_h }\lesssim \|f\|_{L^2_h}^{1-\theta}\|D^+_h f\|_{L^2_h}^{\theta}.
\eeq
\item (Sobolev Inequality)
\beq \label{ineq:Sobolev}
\|f\|_{L^\infty_h}\lesssim \|f\|_{H^1_h}.
\eeq
\end{theoremlist}
\end{lemma}

As a special case of the Gagliardo-Nirenberg inequality when $q=\infty$, we have
\beq \label{ineq:G-N-infty}
\|f\|_{L^{\infty}_h }\le \|f\|_{L^2_h}^{\f{1}{2}}\|D^+_h f\|_{L^2_h}^{\f{1}{2}}.
\eeq
Indeed, for any $f\in L_h^2(h\Z)$ and $x_m\in h\Z$, we observe that
\begin{equation*}\label{eq:sum1}
  h\sum _{h\Z \ni x \ge x_m}\Bigl(f(x+h)\overline{D_h^+f(x)}+ \overline{f(x)}D_h^+f(x)\Bigr)
  = \sum _{h\Z \ni x \ge x_m} \Bigl(|f(x+h)|^2-|f(x)|^2\Bigr)=-|f(x_m)|^2
\end{equation*}
by telescoping. Similarly, we have
\begin{equation*}\label{eq:sum2}
  |f(x_m)|^2= h\sum _{h\Z \ni x < x_m}\Bigl(f(x+h)\overline{D_h^+f(x)}+ \overline{f(x)}D_h^+f(x)\Bigr).
  \end{equation*}
Therefore, using the Cauchy-Schwarz inequality, we have
\[
2|f(x_m)|^2 \leq h\sum_{x\in h\Z} \Bigl(|f(x+h)||D_h^+f(x)|+ |f(x)||D_h^+f(x)|\Bigr)\leq2 \|f\|_{L_h^2}\|D_h^+f\|_{L_h^2}
\]
which deduces \eqref{ineq:G-N-infty}.

\medskip
For a Banach space $X$ with norm $\|\cdot\|_X$ and an interval $I$, we denote by $L^q(I, X)$, $1\leq q < \infty$, the space of all functions $u$ for which
\[
\|u\|_{L^q(I, X)} =\left(\int_I \|u(t)\|_X^q dt\right) ^{\f 1 q}
\]
is finite. If $q=\infty$, use the essential supremum instead.
$\calC(I, X)$ is the space of all continuous functions $u:I \to X$. When $I$ is compact, it is a Banach space with norm
$$
\|u\|_{\calC(I, X)}=\sup _{t\in I}\|u(t)\|_{X}
$$

Now we introduce some properties of the operator $T_{h,r}= e^{ir\Delta_h}$.
The operator $T_{h,r}$ is unitary on $L^2_h(h\Z)$ and, therefore,
\[
\|T_{h,r}f\|_{L^2_h}= \|f\|_{L^2_h}
\]
for all $r\in\R$. Moreover, if $f\in L^2_h$, then
\begin{equation}\label{eq:isometry_H^1}
\|T_{h,r}f\|_{H^1_h}= \|f\|_{H^1_h}
\end{equation}
for all $r\in\R$ since the Fourier transform of $T_{h,r}f$ is
\[
(T_{h,r} f)\char`\^(\xi)=\exp\left(\f{4ir}{h^2} \sin^2\left(\f{h\xi}{2}\right)\right)\hat{f}(\xi).
\]
As a first step, we note that certain space time norm of $T_{h,r}f$ is uniformly bounded in $h$.
\begin{lemma} \label{lem:8}
Let $h\in (0,1]$, then we have
  \beq
  \int_0 ^1 \|T_{h,r}f\|_{L^8_h}^8 dr \lesssim \|f\|^7_{L^2_h} \|D_h ^+f\|_{L^2_h}\notag
  \eeq
for all $f\in L^2_h(h\Z)$.
\end{lemma}
\begin{proof}
By the uniform Strichartz estimate for the discrete Schr\"odinger equation, see the remark below, we have
\beq \label{ineq:stricharz}
 \int_0 ^1 \|T_{h,r}f\|_{L^8_h}^8 dr \lesssim \|f\|_{\dot{H}_h^{\f 1 8}}^8.\notag
\eeq
Moreover, by H\"older's inequality,
\[
\begin{aligned}
\|f\|_{\dot{H}^{\f{1}{8}}_h}^2 = 
\int_{-\f{\pi}{h}}^{\f{\pi}{h}} |\xi|^{\f 14} |\hat{f}(\xi)|^2 d\xi \leq 
\left(\int_{-\f{\pi}{h}}^{\f{\pi}{h}}|\hat{f}(\xi)|^2 d\xi\right)^{\f 7 8}
\left(\int_{-\f{\pi}{h}}^{\f{\pi}{h}} |\xi|^2 |\hat{f}(\xi)|^2 d\xi\right)^{\f 1 8}
=\|f\|_{L^2_h}^{\f 74} \|f\|_{\dot{H}_h^1}^{\f 14}
\,
\end{aligned}
\]
holds. Therefore, combining these two inequalities and \eqref{eq:eq} completes the proof.
\end{proof}

\begin{remark}
 For any $ 2 \le q,r \le \infty$ satisfying
  \[
  \f 3q +\f 1r = \f 12 \, ,
  \]
  the uniform Strichart estimate for the discrete Schr\"odinger equation \eqref{eq:discrete_Schrodinger}
  \[
  \|e^{it\Delta_h}f\|_{L^q_t(\R,L_h^r(h\Z))} \lesssim \|f\|_{\dot{H}_h^{\f 1 q}}
  \]
  holds, see \cite{HY2018} for its proof.
\end{remark}

Denote the nonlinear term of the discrete equation \eqref{eq:dNLS} by
\beq \label{def:Q}
\la Q_h \ra ( f) := \int_0 ^1 T_{h,r}^{-1} (| T_{h,r}f|^{p-1}  T_{h,r}f)dr
\eeq
for $f\in L^2_h(h\Z)$.
We prove an estimate  for $\la Q_h \ra $ which is uniform in $h$.
\begin{lemma}\label{lemma:Hbound}
Let $h\in (0,1]$ and $p>1$. Then
\beq
\| \la Q_h \ra ( f)\|_{H^1_h}\lesssim\|f\|_{H_h^1}^p\notag
\eeq
for all $f\in L^2_h(h\Z)$.
\end{lemma}
\begin{proof}
Using the Minkowski inequality and unitarity of $T_{h,r}$ on $L_h^2(h\Z)$, we have
\[
\begin{aligned}
\|\la Q_h \ra (f)\|_{L^2_h} &\leq \int_0 ^1 \|T_{h,r}^{-1} (|T_{h,r}f|^{p-1} T_{h,r} f)\|_{L^2_h} dr\\
& = \int_0 ^1 \||T_{h,r} f|^{p-1} T_{h,r} f\|_{L^2_h} dr\\
& \leq \int_0 ^1 \|T_{h,r} f\| _{L_h^\infty}^{p-1} \|T_{h,r} f\| _{L_h^2} dr\\
& \lesssim \|  f\| _{L_h^2} \sup_{r\in [0,1]}\|T_{h,r}f\| _{L_h^\infty}^{p-1}.
\end{aligned}
\]
Thus, by the Sobolev inequality \eqref{ineq:Sobolev} and \eqref{eq:isometry_H^1}, we have
\beq\label{ineq:first}
\|\la Q_h \ra (f)\|_{L^2_h} \lesssim \|  f\| _{L_h^2} \sup_{r\in [0,1]}\|T_{h,r}f\| _{H^1_h}^{p-1} \le \|  f\|^{p} _{H_h^1}.
\eeq
Next, we get an estimate for $\|D^+_h(|f|^{p-1}f)\|_{L^2_h}$ instead of $\||f|^{p-1}f\|_{\dot{H}^1_h}$, see \eqref{eq:eq}. Using
\beq \label{ineq:ele}
||z|^{p-1}z-|w|^{p-1}w|\lesssim (|z|^{p-1}+|w|^{p-1})|z-w|
\eeq
for all $z,w\in \C$  and the definition of $D_h^+$, we obtain
\beq
|D_h^+(|f|^{p-1}f)(x)|\lesssim \left(|f(x+h)|^{p-1}+|f(x)|^{p-1}\right) \f{|f(x+h)-f(x)|}{h}\notag
\eeq
and therefore
\[
\begin{aligned}
\|D^+_h(|f|^{p-1}f)\|_{L^2_h} &
\lesssim \| f\|_{L_h^\infty}^{p-1} \|D^+_h f\|_{L_h^2}
\lesssim \|f\| _{H_h^1}^{p}.
\end{aligned}
\]
Using the Minkowski inequality, again, the fact that $T_{h,r}$ and $D_h^+$ commute, and the last inequality, we obtain
\beq\label{ineq:derivative}
\begin{aligned}
\|D^+_h \la Q_h \ra (f)\|_{L^2_h} \leq \int_0 ^1 \|D^+_h (|T_{h,r} f|^{p-1} T_{h,r} f)\|_{L^2_h} dr
\lesssim \|T_{h,r} f\|  _{H_h^1}^{p} = \|f\| _{H_h^1}^{p}.
\end{aligned}
\eeq
Combining \eqref{ineq:first} and \eqref{ineq:derivative} completes the proof.
\end{proof}


\section{global well-posedness}\label{sec:global well-posedness}

In this section, we fix $h\in (0,1]$. First, we consider the local existence of a unique solution for the integral equation of \eqref{eq:dNLS},
\beq\label{eq:duhamel formula}
u_h(t)=e^{i\dav t \Delta_h }\vphi_h+i\int _0 ^t e^{i\dav (t-s)  \Delta_h  } \la Q_h\ra (u_h(s))ds\, .
\eeq
It can be proven by a standard contraction mapping argument but we give a proof for the reader’s convenience.
 \begin{proposition}\label{prop:local}
 Let $\dav\in\R$, $p>1$ and $h\in (0,1]$. For any initial datum $\vphi_h\in L_h^2(h\Z)$, there exists a unique local solution of \eqref{eq:duhamel formula}.
 \end{proposition}
\begin{proof}
Without loss of generality, we assume that $t>0$.
For each $M>0$ and $a>0$, let
\[
B_{M,a}=\{u_h\in L^\infty([0, M], L_h^2(h\Z))  \; : \; \|u_h\|_{L^\infty([0, M],L_h^2)}\leq a\}
\]
be equipped with the distance
$
d(u_h,v_h)=\|u_h-v_h\|_{L^\infty([0, M],L_h^2)}.
$
Let $0\neq \vphi_h\in  L_h^2(h\Z)$  be fixed. Define the map $\Phi$ on $B_{M,a}$ by
\begin{equation}\label{eq:contraction map}
\Phi(u_h)(t)=e^{it \dav \Delta_h }\vphi_h+i\int _0 ^t e^{i(t-s) \dav \Delta_h  } \la Q_h\ra (u_h(s))ds,
\end{equation}
where $\la Q_h\ra $ is defined in \eqref{def:Q}.
Since $\|f\|_{L^\infty_h}\leq h^{-1/2}\|f\|_{L^2_h}$ for all $f\in  L_h^2(h\Z)$,
we see that
\[
\begin{aligned}
\|\la Q_h \ra (u_h)\|_{L^2_h}&\leq 
 \int_0 ^1 \||T_{h,r}u_h|^{p-1} T_{h,r} u_h\|_{L^2_h} dr\\
& \leq  \int_0 ^1 \|T_{h,r} u_h\| _{L_h^\infty}^{p-1} \|T_{h,r} u_h\| _{L_h^2} dr\\
& \leq h^{-\f{p-1}{2}} \| u_h\|^p _{L_h^2}.
\end{aligned}
\]
Thus,  we have
\[
\begin{aligned}
\|\Phi(u_h)(t)\|_{L^2_h}
& \leq \|\vphi_h\|_{L^2_h} + \int _0 ^t  \|\la Q_h\ra (u_h(s))\|_{L^2_h}ds\leq \|\vphi_h\|_{L^2_h} +  h^{-\f{p-1}{2}}M \|  u_h\|^p _{L^\infty([0,M], L_h^2)}
\end{aligned}
\]
for all $0\leq t\leq M$.
Using \eqref{ineq:ele}, the Sobolev inequality \eqref{ineq:Sobolev}, and the unitarity of $T_{h,r}$, we have
\beq
\begin{aligned}\label{ineq:Lips}
d(\Phi(u_h),\Phi(v_h))&\leq \int_0 ^t \|\la Q_h\ra (u_h(s))-\la Q_h \ra(v_h(s))\|_{ L_h^2} ds \\
&\lesssim \int_0 ^M \|u_h(s)-v_h(s)\|_{L^2_h} \int_0 ^1\left(\|T_{h,r}u_h(s)\|_{L^\infty_h}^{p-1} +\|T_{h,r}v_h(s)\|_{L^\infty_h}^{p-1}\right) dr ds \\
& \lesssim h^{-\f{p-1}{2}} M\left(\|u_h\|_{L^\infty([0, M],L^2_h)}^{p-1} +\|v_h\|_{L^\infty([0, M],L^2_h)}^{p-1}\right) d(u_h, v_h).\notag
\end{aligned}
\eeq
Therefore, there exists a positive constant $C$ such that for all $u_h, v_h \in B_{M,a}$,
\beq
\|\Phi(u_h)\|_{L^\infty([0, M], L^2_h)} \leq \|\vphi_h\|_{L^2_h}+ C h^{-\f{p-1}{2}} Ma^p\notag
\eeq
and
\beq
d(\Phi(u_h),\Phi(v_h))\leq C  h^{-\f{p-1}{2}} Ma^{p-1}d(u_h, v_h).\notag
\eeq
Now, we set $a=2\|\vphi_h\|_{L_h^2}$ and choose $M_+>0$ satisfying
\[
C h^{-\f{p-1}{2}}  M_+(2\|\vphi_h\|_{L_h^2})^{p-1} <\f 1 2,
\]
then we obtain that $\Phi$ is a contraction from $B_{M_+,2 \|\vphi_h\|_{L_h^2}}$ into itself. Thus, Banach's contraction mapping theorem shows that there exists a unique solution $u_h$ of \eqref{eq:duhamel formula} in $B_{M_+,2 \|\vphi_h\|_{L_h^2}}$.
Moreover, by \eqref{eq:contraction map}, $u_h\in\calC([0,M_+],L_h^2)$.
\end{proof}

The solution $u_h\in \calC([-M_-,M_+],L_h^2)$ for some $M_{\pm}>0$, given by Proposition \ref{prop:local}, conserves the mass and the energy, that is, for all $t\in[-M_-,M_+]$,
\beq
\|u_h(t)\|^2_{L^2_h}=\|\vphi_h\|^2_{L^2_h}\notag
\eeq
 and
 \beq
 E(u_h(t))= \f \dav 2 \| D_h^+ u_h(t)\|_{L^2_h}^2 - \f{1}{p+1}\int_0 ^1 \|T_{h,r} u_h(t)\|_{L^{p+1}_h}^{p+1} dr=E(\vphi_h).\notag
 \eeq
Note that, unlike the continuous case, the energy is well-defined even on $L^2_h(h\Z)$.
Indeed, using the Gagliardo-Nirenberg inequality \eqref{ineq:G-N} with $\theta=(p-1)/2(p+1)$, we have
\beq \label{ineq:usingGN}
\begin{aligned}
|E(\vphi_h)| & \leq \f {|\dav|}  {2} \|D^+_h \vphi_h\|_{L^2_h}^2 + \f{1}{p+1}\int_0 ^1 \|T_{h,r} \vphi_h\|_{L^{p+1}_h}^{p+1} dr\\
&\lesssim  \|D^+_h \vphi_h\|_{L^2_h}^2 + \int_0 ^1\|T_{h,r}\vphi_h\|_{L^2_h}^{\f{p+3}{2}}\|T_{h,r}D^+ _h \vphi_h\|_{L^2_h}^{\f{p-1}{2}}dr \\
&=   \| D^+_h \vphi_h\|_{L^2_h}^2 + \| \vphi_h\|_{L^2_h}^{\f{p+3}{2}}\|D^+ _h\vphi_h\|_{L^2_h}^{\f{p-1}{2}}
\end{aligned}
\eeq
for any $\vphi_h \in L_h^2(h\Z)$.
Therefore, it follows from \eqref{eq:norm-difference} that the energy is finite.
Moreover, the mass and energy conserve. To show this, as usual, we multiply \eqref{eq:NLS} by $\overline{u_h}$ and $\partial_t \overline{u_h}$, respectively, and then use summation by parts and the elementary facts that the discrete Laplacian $\Delta_h= D_h^- D_h^+ =D_h^+ D_h^-$ on $L^2_h(h\Z)$ and the adjoint operator of $D_h^+$ is $-D_h^-$, where  $ D_h^- $
 denotes the backward difference operator on $L^2_h(h\Z)$
\[
(D_h^-f)(x):= \f{f(x)-f(x-h)}{h}.
\]
Due to the mass conservation law, the time interval where the solution $u_h$ exists can be extended to $\R$, that is, $u_h$ is in $\calC(\R, L^2_h(h\Z))$.

%


To complete Theorem \ref{thm:globalwellposedness}, it remains to show that the map $\vphi_h \mapsto u_h(t)$ is locally Lipschitz continuous on $L^2_h(h\Z)$ by Gronwall’s inequality.
\begin{proposition}
 Let $\dav\in\R$, $p>1$ and $h\in (0,1]$. If $\vphi_h$ and $\psi_h$ are in $L_h^2(h\Z)$, then there exists a positive constant $C_h$ such that
 \beq
\|u_h-v_h\|_{\calC([-T, T], L^2_h)}\le e^{C_h T}\|\vphi_h-\psi_h\|_{L^2_h}\notag
\eeq
for all $T>0$, where $u_h$ and $v_h$ are the global solutions of \eqref{eq:duhamel formula} with the initial data $\vphi_h$ and $\psi_h$, respectively.
\end{proposition}
\begin{proof}
Without loss of generality, we assume that $t>0$. Since
$$
u_h(t)-v_h(t)= e^{i\dav t \Delta_h}(\vphi_h - \psi_h) +i\int_0 ^t e^{i\dav (t-s ) \Delta_h}\left( \la Q_h\ra (u_h(s ))- \la Q_h\ra (v_h(s ))\right)ds ,
$$
we apply the same argument in the proof of Proposition \ref{prop:local} and the mass conservation law to obtain
\begin{align*}
\|u_h(t)-v_h(t)\|_{L_h^2}
&\leq \|\vphi_h - \psi_h\|_{L^2_h}+\int_0 ^t \|\la Q_h\ra (u_h(s ))- \la Q_h\ra (v_h(s ))\|_{L^2_h} ds \\
& \le  \|\vphi_h - \psi_h\|_{L^2_h}+C h^{-\f{p-1}{2}} \left(\|\vphi_h\|_{L^2_h}^{p-1}+\|\psi_h\|_{L^2_h}^{p-1}\right) \int _0 ^t \| u(  s)-v(s )\| _{L^2_h}ds
\end{align*}
for all $t\leq T$.
It follows from Gronwall's inequality that
$$
\|u_h(t)-v_h(t)\|_{L_h^2}\le e^{C h^{-\f{p-1}{2}}(\|\vphi_h\|_{L^2_h}^{p-1}+\|\psi_h\|_{L^2_h}^{p-1})t} \|\vphi_h-\psi_h\|_{L_h^2}\leq e^{C_h T} \|\vphi_h- \psi_h\|_{L_h^2}
$$
for all $t\leq T$, where $C_h=C h^{-\f{p-1}{2}}(\|\vphi_h\|_{L^2_h}^{p-1}+\|\psi_h\|_{L^2_h}^{p-1})$.
\end{proof}

\section{Continuum limit}\label{sec:continuum limit}
In this section, we directly compare $u(t)$ and $p_hu_h(t)$ to get a strong convergence in $L^2(\R)$ as in \cite{HY2019}, where $u$ and $u_h$ are the global solutions of \eqref{eq:NLS} and \eqref{eq:dNLS} with the initial data $\vphi$ and $\vphi_h$, respectively. Recall that $\vphi_h$ is the discretization of $\vphi \in H^1(\R)$. First, we collect some  properties, in the form we need, of discretization and linear interpolation from \cite{HY2019, KLS2013}.
\begin{lemma} \label{lem:boundedness1} 
 \begin{theoremlist}
 \item
If $f\in H^1(\R)$, then its discretization $f_h$ satisfies
\beq \label{boundedness1}
\|f_h\|_{L^2_h}\leq \|f\|_{L^2} \;\; \text{and} \;\; \|D_h^+f_h\|_{L^2_h}\leq \|f'\|_{L^2}.\notag
\eeq
 \item
 If $f_h\in L^2_h(h\Z)$, then
 \beq \label{boundedness2}
\|p_hf_h\|_{H^1}\lesssim \|f_h\|_{H_h^1}.\notag
\eeq
\end{theoremlist}
\end{lemma}

\begin{lemma} \label{lem:linear} 
 \begin{theoremlist}
\item
 If $f_h$ is the discretization of $f\in H^1(\R)$, then
 \beq\label{eq:linear0}
 \|p_h f_h- f\|_{L^2} \lesssim   h\|f\|_{H^1}.\notag
 \eeq
 \item
  If $f\in H^1(\R)$ and $f_h\in L^2_h(h\Z)$, then
\beq \label{eq:linear1}
\|p_h e^{it\Delta_h }f_h- e^{it\partial_x^2}f\|_{L^2}\lesssim h^{\f 1 2}|t| \bigg\{ \|f_h\|_{H^1_h}+ \|f\|_{H^1} \bigg\} + \|p_h f_h - f\|_{L^2}.\notag
\eeq
\end{theoremlist}
\end{lemma}

\begin{lemma} \label{lem:distributive} 
  If $p>1$, then
 \beq\label{eq:distributive}
\|p_h (|f_h|^{p-1}f_h)- |p_hf_h|^{p-1}p_hf_h \|_{L^2}\lesssim h \|f_h\|^{p-1}_{L^\infty_h}\|f_h\|_{H^1_h}\notag
\eeq
for all $f_h\in L^2_h(h\Z)$.
\end{lemma}

Now recall that, {\it for each $h \in(0,1]$}, the global solution $u_h$ of \eqref{eq:dNLS} exists in $\calC(\R, L^2_h(h\Z))$, which is guaranteed by Theorem \ref{thm:globalwellposedness}. Moreover, it follows from \eqref{eq:relation L^2 H^1} that $\|u_h(t)\|_{H_h^1}$ is finite  {\it for each  $t\in\R$ and each $h \in(0,1]$}. Furthermore, it is bounded in $H^1_h$ uniformly in $t$ and $h$, which plays a crucial role in proving Theorem \ref{thm:continuum limit}.

First, we give the uniform $H^1_h$-bound in the case of $\dav\neq0$ which comes from the conservation laws.
\begin{proposition}
[Uniform $H^1_h$-bound for $\dav\neq 0$]\label{prop:uniform}
Let $\dav\neq 0$, $h\in (0,1]$ and $p$ satisfy \eqref{ass:p}. Given $\vphi\in H^1(\R)$,
let $u_h\in \calC(\R, L^2_h(h\Z))$ be the global solution of \eqref{eq:dNLS} whose initial datum $\vphi_h$  is the discretization of $\vphi$. Then there exists a positive constant $C$ depending only on $\dav, p$, and $\|\vphi\|_{H^1}$ such that
\beq \label{sup}
\sup_{h\in (0, 1]}\sup _{t\in \R} \|u_h(t)\|_{H^1_h}\le C.\notag
\eeq
\end{proposition}
\begin{proof}
First, recall that the $H_h^1$-norm is equivalent to the norm in \eqref{eq:equivalent norm}, i.e., for any $f \in L_h^2(h\Z)$
\[
\|f\|_{H_h^1} \sim \left(\|f\|^2_{L^2_h}+\|D^+_h f\|^2_{L^2_h}\right)^{1/2}.
\]
It is easy to see that for all $t\in\R$ and $h\in(0,1]$
\[
\|u_h(t)\|_{L_h^2} \le \|\vphi\|_{L^2}
\]
by the mass conservation law and Lemma \ref{lem:boundedness1} (i).
Next, note that the energy $E(\vphi_h)$ is bounded uniformly in $h$ since
\beq \label{energy:constant}
|E(\vphi_h)|
\lesssim \f{|\dav|}{2}  \| D^+_h \vphi_h\|_{L^2_h}^2 + \f{1}{(p+1)}\| \vphi_h\|_{L^2_h}^{\f{p+3}{2}}\|D^+ _h\vphi_h\|_{L^2_h}^{\f{p-1}{2}}
 \lesssim \|\vphi\|_{H^1}^2 +  \|\vphi\|_{H^1}^{p+1},
\eeq
where we used \eqref{ineq:usingGN} and  Lemma \ref{lem:boundedness1} (i).

When $\dav<0$, we use the energy conservation law to get
\[
\|D^+_h u_h(t)\|_{L^2_h}^2 \leq \f{2|E(\vphi_h)|}{|\dav|}
\]
for all $t\in \R$, which together with \eqref{energy:constant} completes the proof for this case.

Now we consider the case $\dav>0$.
 When $1<p<7$, we use the H\"older inequality 
  and the unitarity of $T_{h,r}$ on $L^2_h$ to see
\[
\begin{aligned}
\|T_{h,r}u_h(t)\|_{L^{p+1}_h}^{p+1}
& = \| |T_{h,r}u_h(t)|^{\f{7-p}{3}} |T_{h,r}u_h(t)|^{\f{4(p-1)}{3}}\|_{L^{1}_h}\\
&\leq  \|T_{h,r}u_h(t)\|_{L^{2}_h}^{\f{7-p}{3}} \|T_{h,r}u_h(t)\|_{L^{8}_h}^{\f{4(p-1)}{3}}\\
& = \|u_h(t)\|_{L^{2}_h}^{\f{7-p}{3}} \|T_{h,r}u_h(t)\|_{L^{8}_h}^{\f{4(p-1)}{3}}.
\end{aligned}
\]
Applying the H\"older inequality in $r$-integral with $6/(7-p)$ and $6/(p-1)$ and Lemma \ref{lem:8}, we have
\[
\int_0 ^1\|T_{h,r}u_h(t)\|_{L^{8}_h}^{\f{4(p-1)}{3}}dr \leq \left(\int_0 ^1\|T_{h,r}u_h(t)\|_{L^{8}_h}^{8}dr\right)^{\f{p-1}{6}}
 \lesssim \|u_h(t)\|_{L^{2}_h}^{\f {7(p-1)}{6}} \|D_h^+ u_h(t)\|_{L^{2}_h}^{\f{p-1}{6}}\,.
\]
Using the last two inequalities and mass conservation, we get
\beq\label{p:small}
\int_0 ^1\|T_{h,r}u_h(t)\|_{L^{p+1}_h}^{p+1} dr \lesssim  \|\vphi_h\|_{L^{2}_h}^{\f {5p+7}{6}}\|D_h^+ u_h(t)\|_{L^{2}_h}^{\f{p-1}{6}}.
\eeq
Next we consider the case $7\leq p<9$. Then
\[
\int_0 ^1\|T_{h,r}u_h(t)\|_{L^{p+1}_h}^{p+1} dr \leq   \sup_{0\leq r \leq 1} \|T_{h,r}u_h(t)  \|_{L^{\infty}_h}^{p-7}
\int_0 ^1\|T_{h,r}u_h(t)\|_{L^{8}_h}^{8}dr.
\]
We recall that the Gagliardo-Nirenberg inequality \eqref{ineq:G-N} gives
\[
\|f\|_{L^\infty_h}\lesssim \|f\|_{L^{2}_h}^{\f{1}{2}}\|D_h^+ f\|_{L^{2}_h}^{\f{1}{2}}
\]
for $f\in L^2_h(h\Z)$.
Thus, using this, Lemma \ref{lem:8} and mass conservation, we have
\beq\label{p:large}
\int_0 ^1\|T_{h,r}u_h(t)\|_{L^{p+1}_h}^{p+1} dr\lesssim   \|\vphi_h\|_{L^{2}_h}^{\f {p+7}{2}} \|D_h^+ u_h(t)\|_{L^{2}_h}^{\f{p-5}{2}}.
\eeq
Therefore, for any $1<p<9$, using the energy conservation laws, \eqref{p:small}, \eqref{p:large}, \eqref{energy:constant} and Lemma \ref{lem:boundedness1}(i), we have
\begin{equation}\label{global}
\begin{aligned}
\|D^+_h u_h(t)\|_{L^2_h}^2 &=\f{2E(\vphi_h)}{\dav}+ \f{2}{(p+1)\dav}\int_0 ^1\|T_{h,r}u_h(t)\|_{L^{p+1}_h}^{p+1} dr \\
&\lesssim |E(\vphi_h)| + \|\vphi_h\|_{L^{2}_h}^{\kappa_1}\|D_h^+ u_h(t)\|_{L^{2}_h}^{\kappa_2} \\
& \lesssim \left( \|\vphi\|_{H^1}^2 +\|\vphi\|_{H^1}^{p+1} \right) + \|\vphi\|_{L^2}^{\kappa_1}\|D^+ _h u_h(t)\|_{L^2_h}^{\kappa_2},
   \end{aligned}
\end{equation}
where we used $\kappa_1$ and $\kappa_2$ instead of the exponents in \eqref{p:small} and \eqref{p:large}. Noting $0<\kappa_2<2$,  we have a constant $C=C(\dav, p, \|\vphi\|_{H^1_h})$ such that
\[
\sup_{t\in\R}\|D^+_h u_h(t)\|_{L^2_h} \leq  C
\]
which completes the proof.
\end{proof}

\begin{remark}
When $\dav>0$, if we only consider the case $1<p<5$, the uniform $H^1_h$-bound on solutions immediately follows from the mass and energy conservation laws and the Gagliardo-Nirenberg inequality \eqref{ineq:G-N} as in \eqref{ineq:usingGN} to get
\[
\|D^+_h u_h(t)\|_{L^2_h}^2 \lesssim
 \left( \|\vphi\|_{H^1}^2 +\|\vphi\|_{H^1}^{p+1} \right) + \|\vphi_h\|_{L^2_h}^{\f{p+3}{2}}\|D^+ _h u_h(t)\|_{L^2_h}^{\f{p-1}{2}}
\]
which is \eqref{global} with $0<\kappa_2=\f{p-1}{2}<2$.
\end{remark}

In the case of $\dav=0$, to obtain the uniform $H^1_h$-bound,
we state a well-known generalization of Gronwall’s inequality in \cite{Bihari}, see also \cite{LaSalle}.
\begin{lemma}\label{lem:Gronwall}
Let $v:[a, b]\to \R_+$ be a continuous function that satisfies the inequality
\beq
v(t)\leq M +\int_a^t f(s)\omega(v(s))ds, \; t\in [a,b],\notag
\eeq
where $M\geq 0$, $f:[a,b]\to \R_+$ is continuous and $\omega:\R_+\to \R_+$ is a continuous and monotone increasing function with $\omega(s)>0$ for $s>0$. Then
\beq
v(t)\leq G^{-1}\left( G(M)+\int_a^t f(s)ds\right), \; t\in [a,b]\notag
\eeq
where $G:\R\to \R$ is given by
\[
G(x):= \int_{x_0}^x\f{1}{\omega(s)}ds, \; x\in \R.
\]
\end{lemma}

\begin{proposition}
[Uniform $H^1_h$-bound for $\dav=0$]\label{prop:uniform2}
Let $\dav=0$, $h\in (0,1]$ and $1<p<5$. Given $\vphi\in H^1(\R)$,
let $u_h\in \calC(\R, L^2_h(h\Z))$ be the global solution of \eqref{eq:dNLS} whose initial datum $\vphi_h$  is the discretization of $\vphi$. Then, for any $0<T<\f{2}{p-1}(\|\vphi\|_{L^2} \|\vphi'\|_{L^2})^{-\f{p-1}{2}}$,
the inequality
\beq \label{est:dav=0}
\sup_{h\in (0, 1]}
\sup _{t\in [-T,T]}
\|D_h^+ u_h(t)\|_{L_h^2} \leq  \left(\|\vphi'\|_{L^2}^{-\f{p-1}{2}} -\f{p-1}{2} \|\vphi\|_{L^2}^{\f{p-1}{2}}T\right)^{-\f{2}{p-1}}\notag
\eeq
holds.
\end{proposition}
\begin{proof}
We consider $t>0$ only.
From \eqref{eq:duhamel formula}, we have
\beq
\|D_h^+ u_h(t)\|_{L_h^2} \leq \|D_h^+ \vphi_h\|_{L_h^2}+ \int_0 ^t \|D_h^+ \la Q_h\ra (u_h(s))\|_{L_h^2}ds. \notag 
\eeq
Note that by the Gagliardo-Nirenberg inequality \eqref{ineq:G-N-infty}
\[
\begin{aligned}
\|D_h^+ \la Q_h\ra (u_h(t))\|_{L_h^2} & \leq \int_0 ^1\|T_{h,r} u_h(t)\|_{L_h^\infty}^{p-1} \|D_h^+ T_{h,r}u_h(t)\|_{L^2_h} dr\\
&\le \|u_h(t)\|_{L^2_h}^{\f{p-1}{2}} \|D_h^+ u_h(t)\|_{L^2_h}^{\f{p+1}{2}}.
\end{aligned}
\]
Using this, Lemma \ref{lem:boundedness1}, and mass conservation, we have
\beq
\begin{aligned}
\|D_h^+ u_h(t)\|_{L_h^2} \leq
 \|\vphi'\|_{L^2} + \|\vphi\|_{L^2}^{\f{p-1}{2}}\int_0 ^t \|D_h^+ u_h(s)\|_{L_h^2}^{\f{p+1}{2}}ds.\notag
\end{aligned}
\eeq
Then applying Lemma \ref{lem:Gronwall} completes the proof
\end{proof}

Now we are ready to prove Theorem \ref{thm:continuum limit} by directly comparing
\beq
p_h u_h(t)=p_h e^{i\dav t\Delta_h}\vphi_h-i\int_0 ^t p_h e^{i\dav (t-s)\Delta_h}
\la Q_h\ra(u_h(s))ds\notag
\eeq
and
\beq
u(t)=e^{i\dav t\partial_x^2}\vphi-i\int_0 ^t e^{i\dav(t-s)\partial_x^2}
\la Q\ra(u(s))ds,\notag
\eeq
where $\la Q_h \ra$ is given in \eqref{def:Q} and
\[
\la Q \ra ( f) := \int_0 ^1 T_{r}^{-1} (| T_{r}f|^{p-1}  T_{r}f)dr
\]
for $f\in H^1(\R)$.
Recall that $u_h$ and $u$ are the global solutions of \eqref{eq:dNLS} and \eqref{eq:NLS} with the initial data $\vphi_h$ and $\vphi$, respectively. Here, $u_h$ is not the discretization of $u$ while $\vphi_h$ is the discretization of $\vphi$.
\begin{proof}[Proof of Theorem \ref{thm:continuum limit}]
We prove the case when $\dav \neq 0$ since the case of $\dav=0$ can be proven analogously by Proposition \ref{prop:uniform2} instead of Proposition \ref{prop:uniform}.
Fix $0<T< \infty$, we consider positive times only and write the difference of $p_h u_h(t)$ and $u(t)$ as
\beq
p_h u_h(t)-u(t)=I_1(t) -i(I_2(t)+  I_3(t) +I_4(t)),\notag
\eeq
where
\[
\begin{aligned}
& I_1(t) :=p_h e^{it\dav\Delta_h}\vphi_h- e^{it\dav\partial_x^2}\vphi,\\
& I_2(t):=\int_0 ^t  \big(p_h e^{i(t-s)\dav\Delta_h}- e^{i(t-s)\dav\partial_x^2} p_h\big)\la Q_h\ra (u_h(s))ds,\\
& I_3(t):=\int_0 ^t   e^{i(t-s)\dav\partial_x^2} \bigg(p_h\la Q_h\ra (u_h(s)) - \la Q\ra(p_h u_h(s))\bigg) ds,\\
& I_4(t):=\int_0 ^t   e^{i(t-s)\dav\partial_x^2} \bigg( \la Q\ra(p_h u_h(s)) - \la Q\ra (u(s))\bigg) ds
\end{aligned}
\]
for $0<t<T$.
It follows from 
Lemmas \ref{lem:linear} (i) and \ref{lem:boundedness1} (i) that
\beq \label{I1}
\begin{aligned}
\|I_1(t)\|_{L^2}
\lesssim h^{\f 1 2}t(\|\vphi_h\|_{H_h^1}+ \|\vphi\|_{H^1})+ \|p_h \vphi_h-\vphi\|_{L^2}
\lesssim h^{\f 1 2}(1+t) \|\vphi\|_{H^1}.
\end{aligned}
\eeq
For $I_2(t)$,  we apply 
Lemmas \ref{lem:linear} (ii), \ref{lem:boundedness1} (ii), and \ref{lemma:Hbound} to obtain
\beq\label{I2}
\begin{aligned}
\|I_2(t)\|_{L^2} & \leq\int_0 ^t \left\| \big(p_h e^{i(t-s)\dav\Delta_h}- e^{i(t-s)\dav\partial_x^2} p_h\big)\la Q_h\ra (u_h(s))\right\|_{L^2}ds \\
& \lesssim \int_0 ^t h^{\f 1 2}|t-s|\left[\|\la Q_h\ra (u_h(s))\|_{H_h^1} + \|p_h \la Q_h\ra (u_h(s))\|_{H^1}\right]ds\\
& \lesssim  h^{\f 1 2}t  \int_0 ^t \| \la Q_h \ra (u_h(s))\|_{H^1_h}ds\\
& \lesssim h^{\f 1 2}t \int_0 ^t \|u_h(s)\|^p_{H^1_h}ds.
\end{aligned}
\eeq
For $I_4(t)$,
note from \eqref{ineq:ele}, the embedding $H^1(\R) \hookrightarrow L^\infty(\R)$, and the unitarity of $T_r$ that
 \[
\begin{aligned}
 &\int_0 ^1\|   |T_r (p_h u_h(s))|^{p-1}T_r (p_h u_h(s)) - |T_r u(s)|^{p-1} T_r u(s)\|_{L^2} dr\\
& \lesssim \int_0 ^1 \left(\|T_r (p_h u_h(s))\|_{L^\infty}^{p-1}+ \|T_r u(s)\|_{L^\infty}^{p-1} \right) \|p_h u_h(s)- u(s)\|_{L^2}dr \\
& \leq
\int_0 ^1 \left(\|T_r (p_h u_h(s))\|_{H^1}^{p-1}+ \|T_r u(s)\|_{H^1}^{p-1} \right) \|p_h u_h(s)- u(s)\|_{L^2}dr\\
& =  \left(\|p_h u_h(s)\|_{H^1}^{p-1}+ \| u(s)\|_{H^1}^{p-1} \right) \|p_h u_h(s)- u(s)\|_{L^2}.
\end{aligned}
\]
Thus, using the Minkowski's inequality and Lemma \ref{lem:boundedness1} (ii),
we have
\beq \label{I3}
\begin{aligned}
\|I_4(t)\|_{L^2} 
&\leq\int_0 ^t \int_0 ^1 \|   |T_r (p_h u_h(s))|^{p-1}T_r (p_h u_h(s)) - |T_r u(s)|^{p-1} T_r u(s)\|_{L^2} drds\\
& \lesssim \left(\|u_h\|_{L^\infty([0,T], H^1_h)}^{p-1}+ \|u\|_{L^\infty([0,T], H^1)}^{p-1}\right) \int_0 ^t
\|p_h u_h(s)- u(s)\|_{L^2}  ds.
\end{aligned}
\eeq
For $I_3(t)$, we decompose
$$
I_3(t)=\int_0 ^t   e^{i(t-s)\dav\partial_x^2} \bigg( I_{3,1}(s) + I_{3,2}(s) + I_{3,3}(s)\bigg)ds,
$$
 where
\beq
\begin{aligned}
I_{3,1}(s) &:=\int_0 ^1 p_h T_{h,r}^{-1}\bigg(|T_{h,r}u_h(s)|^{p-1}T_{h,r}u_h(s)\bigg) - T_r ^{-1}\bigg(p_h(|T_{h,r}u_h(s)|^{p-1}T_{h,r}u_h(s))\bigg) dr,\\ \notag
I_{3,2}(s) &:=\int_0 ^1 T_r ^{-1}\bigg(p_h (|T_{h,r}u_h(s)|^{p-1}T_{h,r}u_h(s))  - |p_h  T_{h,r} u_h(s)|^{p-1}p_h  T_{h,r}  u_h(s)\bigg) dr, \\\notag
I_{3,3}(s) &:=\int_0 ^1 T_r ^{-1}\bigg(|p_h  T_r u_h(s)|^{p-1}p_h  T_r  u_h(s) - |T_r (p_h u_h)(s)|^{p-1}T_r (p_h u_h)(s)\bigg) dr\notag
\end{aligned}
\eeq
for $0<s<t$.
We claim that
\[
\|I_{3,1}(s) + I_{3,2}(s) + I_{3,3}(s) \|_{L^2} \lesssim h^{\f 1 2} \|u_h(s)\|_{H^1_h}^p.
\]
Indeed, it follows from Lemma \ref{lem:distributive} and Sobolev inequality \eqref{ineq:Sobolev} that
\[
\begin{aligned}
\|I_{3,2}(s) \|_{L^2} &\leq \int_0 ^1\|p_h (|T_{h,r}u_h(s)|^{p-1}T_{h,r}u_h(s)) - |p_h  T_{h,r} u_h(s)|^{p-1}p_h  T_{h,r}  u_h(s)\|_{L^2} dr \\
&\lesssim  h \int_0 ^1\|T_{h,r}u_h(s)\|_{L_h ^\infty}^{p-1} \|T_{h,r}u_h(s)\|_{H^1_h}dr\\
& \lesssim  h \int_0 ^1\|T_{h,r}u_h(s)\|_{H^1_h}^{p} dr = h \|u_h(s)\|_{H^1_h}^p.
\end{aligned}
\]
For $I_{3,3}(s)$, we apply Lemma \ref{lem:linear} (ii) that
\[
\|p_hT_{h,r}u_h(s) - T_r (p_hu_h)(s)\|_{L^2} \lesssim h^{\f 1 2} |r| (\|u_h(s)\|_{H^1_h} +\|p_h u_h(s)\|_{H^1}) \lesssim h^{\f 1 2} |r| \|u_h(s)\|_{H^1_h}.
\]
Then, by the Sobolev inequality \eqref{ineq:Sobolev}, the embedding $H^1(\R) \hookrightarrow L^\infty(\R)$, and Lemma \ref{lem:linear} (ii), we have
\[
\begin{aligned}
\|I_{3,3}(s) \|_{L^2} &\leq \int_0 ^1\| |p_hT_{h,r}u_h(s)|^{p-1}p_hT_{h,r}u_h(s) - | T_r (p_hu_h)(s)|^{p-1}  T_r  (p_hu_h)(s)\|_{L^2} dr \\
&\lesssim \int_0 ^1\left(\|p_hT_{h,r}u_h(s)\|^{p-1}_{L^\infty} + \|T_r (p_hu_h)(s)\|^{p-1} _{L^\infty}  \right) \|p_hT_{h,r}u_h(s) - T_r (p_hu_h)(s)\|_{L^2}dr\\
& \lesssim h^{\f{1}{2}} \int_0 ^1 r \|u_h(s)\|^p_{H^1_h} dr
\leq h^{\f{1}{2}} \|u_h(s)\|^p_{H^1_h}.
\end{aligned}
\]
Since $|T_{h,r}u_h(s)|^{p-1}T_{h,r}u_h(s)\in H^1_h(\R)$ and $\||T_{h,r}u_h(s)|^{p-1}T_{h,r}u_h(s)\|_{H^1_h}\lesssim \|u_h(s)\|_{H^1_h}^p$,
by Lemmas \ref{lem:linear} (ii) and \ref{lem:boundedness1} (ii), we have
\[
\begin{aligned}
&\|I_{3,1}(s) \|_{L^2}\\
& \leq  \int_0 ^1
\left\|p_h T_{h,r}^{-1}\bigg(|T_{h,r}u_h(s)|^{p-1}T_{h,r}u_h(s)\bigg) - T_r ^{-1}\bigg(p_h\big(|T_{h,r}u_h(s)|^{p-1}T_{h,r}u_h(s)\big)\bigg)\right\|_{L^2} dr \\
& \lesssim h^{\f{1}{2}}\int_0 ^1  r \||T_{h,r}u_h(s)|^{p-1}T_{h,r}u_h(s)\|_{H_h ^1}dr \\
& \lesssim h^{\f{1}{2}} \|u_h(s)\|^p_{H^1_h}.
\end{aligned}
\]
Thus,
\beq
\|I_3(t)\|_{L^2_h}
\lesssim
h^{\f{1}{2}}  \int_0 ^t  \|u_h(s)\|_{H^1_h}^p ds \notag
\eeq
since $0< h\leq 1$.
Combining this, \eqref{I1}, \eqref{I2}, and \eqref{I3}, we have
\[
\begin{aligned}
\|p_h u_h(t)-u(t)\|&_{L^2}\lesssim   \; h^{\f 1 2}( 1+t) \bigg(\|\vphi\|_{H^1} + T\|u_h\|_{L^\infty([0,T],H_h^1)}^p\bigg) \\
&+ \left(\|u_h\|_{L^\infty([0,T], H^1_h)}^{p-1}+ \|u\|_{L^\infty([0,T], H^1)}^{p-1}\right)\int_0 ^t
\|p_h u_h(s)- u(s)\|_{L^2}  ds
\end{aligned}
\]
for all $0< t\leq T$ and $h\in (0,1]$.
Finally, by Gronwall’s inequality, we have
\[
\begin{aligned}
\|p_h u_h(t)-u(t)\|_{L^2} \lesssim & h^{\f 1 2}( 1+t) \bigg(\|\vphi\|_{H^1} + T\|u_h\|_{L^\infty([0,T],H_h^1)}^p\bigg)\\
&\times \exp\left\{ {t\left(\|u_h\|_{L^\infty([0,T], H^1_h)}^{p-1}+ \|u\|_{L^\infty([0,T], H^1)}^{p-1}\right)}\right\}
\end{aligned}
\]
for all $0< t\leq T$.
Thus, using Proposition \ref{prop:uniform} and the fact that $u\in \calC(\R, H^1(\R))$, we have two positive constants $A$ and $B$, independent of $h$ and  $T$, satisfying
\beq
\sup _{0< t\leq T}\|p_h u_h(t)-u(t)\|_{L^2} \le  A h^{\f 1 2}e^ {BT}.\notag
\eeq
\end{proof}

\vspace{5mm}

\appendix
\setcounter{section}{0}
\renewcommand{\thesection}{\Alph{section}}
\renewcommand{\theequation}{\thesection.\arabic{equation}}
\renewcommand{\thetheorem}{\thesection.\arabic{theorem}}

\noindent

\textbf{Acknowledgements: }
The authors are supported by the National Research Foundation of Korea(NRF) grants funded by the Korean government (MSIT) NRF-2020R1A2C1A01010735 and (MOE) NRF-2021R1I1A1A01045900.

%
%
%

 \bibliographystyle{abbrv}
 \bibliography{continuumlimit_bibfile}

\end{document}